\documentclass[10pt]{amsart}
\usepackage{enumerate,amsmath,amssymb,latexsym,
amsfonts, amsthm, amscd, MnSymbol}

\theoremstyle{plain}

\newtheorem{theorem}{Theorem}

\numberwithin{equation}{section}

\newcommand{\ii}{{\rm i}}
\newcommand{\dd}{{\rm dn}_3}
\newcommand{\aaa}{\theta}
\newcommand{\sss}{{\rm sn}_3}
\newcommand{\ccc}{{\rm cn}_3}

\begin{document}

\title {The elliptic function ${\rm dn}_3$ of Shen}

\date{}

\author[P.L. Robinson]{P.L. Robinson}

\address{Department of Mathematics \\ University of Florida \\ Gainesville FL 32611  USA }

\email[]{paulr@ufl.edu}

 \keywords{}

\begin{abstract}

We analyze the elliptic function ${\rm dn}_3$ introduced by Li-Chien Shen, contributing to the Ramanujan theory of elliptic functions in signature three. A famous hypergeometric identity emerges from our analysis. 

\end{abstract}

\maketitle

\medbreak

\section*{Introduction}

\medbreak 

Li-Chien Shen [5] adapted one of the classical approaches to the Jacobian elliptic functions and thereby produced a ${\rm dn}$-counterpart $\dd$ as a contribution to the Ramanujan theory of elliptic functions in signature three. Shen analyzed several structural aspects of the elliptic function $\dd$: in particular, he located its poles and expressed it in terms of theta functions; among the by-products of his analysis is a proof of the cubic identity `$a^2 = b^2 + c^2$' of the brothers Borwein. 

\medbreak 

In the present paper, we present a relatively detailed account of the elliptic function $\dd$: as well as locating its poles, we also give concrete expressions for its critical values and its fundamental periods; as a by-product of our analysis, we encounter a famous hypergeometric identity. Throughout, we choose to avoid the use of theta functions: in doing so, we forgo some of the concomitant benefits; as compensation, our arguments are perhaps more direct. 

\medbreak 

\section*{Definition} 

\medbreak 

Fix moduli $\kappa \in (0, 1)$ and $\lambda \in (0, 1)$ that are complementary in the usual sense that $\kappa^2 + \lambda^2 = 1$. An analytic  function $f$ is defined in a neighbourhood of the origin by the rule that if $T$ is near $0$ then 
$$f (T) = \int_0^T F(\tfrac{1}{3}, \tfrac{2}{3} ; \tfrac{1}{2} ; \kappa^2 \sin^2 t) \, {\rm d}t.$$ 
If $T$ is taken to be real, then of course it need not be restricted as to size; if $T$ is complex, then we need only restrict $T$ so that $\kappa \, |\sin| < 1$  on the closed disc of radius $|T|$ about $0$. Note that 
$$f' (T) = F(\tfrac{1}{3}, \tfrac{2}{3} ; \tfrac{1}{2} ; \kappa^2 \sin^2 T)$$
so $f'(0) = 1 \neq 0$ in particular. It follows that $f$ has an analytic inverse in a neighbourhood of $0$. We shall denote this inverse by $\phi$: thus, $\phi(0) = 0$ and if $u$ is near $0$ then 
$$u = \int_0^{\phi(u)} F(\tfrac{1}{3}, \tfrac{2}{3} ; \tfrac{1}{2} ; \kappa^2 \sin^2 t) \, {\rm d}t.$$ 
An auxiliary analytic function $\psi$ is defined on a (possibly smaller) neighbourhood of $0$ by $\psi(0) = 0$ and the requirement 
$$\sin \psi = \kappa \, \sin \phi.$$ 

\medbreak 

Now, the elliptic function $\dd$ of Shen is initially defined to be the derivative of the local inverse $\phi$: that is, 
$$\dd = \phi '.$$
The precise neighbourhood of $0$ on which $\dd$ is initially defined is unimportant: for the present,  it may be assumed that all calculations are performed in an open disc $D$ about $0$ on which $\psi$, $\phi$ and (of course) $\dd$ are analytic. As we shall see, $\dd$ turns out to be the restriction of a uniquely-defined elliptic function; this elliptic extension is $\dd$ proper. 

\medbreak 

\section*{Algebra} 

\medbreak 

We shall establish the elliptic nature of $\dd$ by confirming that it satisfies a specific first-order differential equation. The derivation of this differential equation is facilitated by the introduction of two new functions via composition: 
$$\ccc = \cos \circ \phi \; \; {\rm and} \; \; \sss = \sin \circ \phi$$ 
initially defined on the disc $D$. These functions may be viewed as analogues of the classical Jacobian elliptic functions ${\rm cn}$ and ${\rm sn}$; they certainly satisfy the `Pythagorean' identity 
$$\ccc^2 + \sss^2 = 1.$$
The function $\dd$ was originally introduced as a substitute for the third Jacobian elliptic function ${\rm dn}$; in place of the Jacobian identity ${\rm dn}^2 + \kappa^2 {\rm sn}^2 = 1$ we have the following. 

\medbreak 

\begin{theorem} \label{Py}
The functions $\dd$ and $\sss$ satisfy the equivalent relations 
$$4(1 - \kappa^2 \sss^2) = \dd^3 + 3 \, \dd^2$$
and 
$$4 \kappa^2 \sss^2 = (1 - \dd)(2 + \dd)^2.$$
\end{theorem} 

\begin{proof} 
As $\phi$ is a local inverse to $f$ it follows that 
$$\dd = \phi' = \frac{1}{f' \circ \phi} = \frac{1}{F(\tfrac{1}{3}, \tfrac{2}{3} ; \tfrac{1}{2} ; \kappa^2 \sin^2 \phi)} = \frac{1}{F(\tfrac{1}{3}, \tfrac{2}{3} ; \tfrac{1}{2} ; \sin^2 \psi)}.$$
On account of the standard hypergeometric identity 
$$F(\tfrac{1}{3}, \tfrac{2}{3} ; \tfrac{1}{2} ; \sin^2 t) = \frac{\cos \frac{1}{3} t}{\cos t}$$
we deduce that 
$$\dd = \frac{\cos \psi}{\cos \frac{1}{3} \psi}.$$
Trigonometric triplication yields
$$\cos \psi = 4 \cos^3 \tfrac{1}{3} \psi - 3 \cos \tfrac{1}{3} \psi$$
whence substitution of $\cos \psi / \dd$ for $\cos \frac{1}{3} \psi$ followed by cancellation of $\cos \psi$ and multiplication throughout by $\dd^3$ lead after rearrangement to 
$$\dd^3 + 3 \, \dd^2 = 4 \cos^2 \psi = 4 (1 - \sin^2 \psi) = 4 (1 - \kappa^2 \sin^2 \phi) = 4 (1 - \kappa^2 \sss^2).$$
This establishes the first claimed equation; the second is a matter of factorization. 
\end{proof} 

\medbreak 

Note that the `Pythagorean' identity allows us to deduce from this the equation  
$$4 \kappa^2 \ccc^2 = \dd^3 + 3 \, \dd^2 - 4 \lambda^2.$$ 

\medbreak 

\section*{Ellipticity}

\medbreak 

Regarding derivatives, the counterparts  
$$\sss ' = \ccc \, \phi \,' = \ccc \, \dd$$ 
$$\ccc ' = - \sss \, \phi \,' = - \sss \, \dd$$
of the classical Jacobian formulae follow at once from the definitions. The derivative of $\dd$ itself deviates from the classical form as follows. 

\medbreak 

\begin{theorem} \label{d'}
The function $\dd$ satisfies the first-order differential equation 
$$(\dd ')^2 = \frac{4}{9} (1 - \dd)\,(\dd^3 + 3 \, \dd^2 - 4 \lambda^2).$$
\end{theorem} 

\begin{proof} 
Differentiate through the first equation of Theorem \ref{Py}: in view of $\sss ' = \ccc \, \dd$ as noted above, it follows after cancellation of $\dd$ that 
$$- 8 \, \kappa^2 \, \sss \, \ccc = 3 \, (\dd + 2) \, \dd '$$
and thence that 
$$64 \, \kappa^4 \, \sss^2 \, \ccc^2 = 9 \, (\dd + 2)^2 \, (\dd ')^2.$$
Here, substitute for $4 \kappa^2 \sss^2$ and $4 \kappa^2 \ccc^2$ directly from Theorem \ref{Py}. Finally, cancel $(\dd + 2)^2$ throughout to conclude the argument. 
\end{proof} 

\medbreak 

Armed with this differential equation, we are now in a position to recognize that $\dd$ is the restriction of an elliptic function, which we identify in terms of its coperiodic Weierstrass function. 

\medbreak 

\begin{theorem} \label{ddwp}
The function $\dd$ satisfies 
$$(1 - \dd)\, (\frac{1}{3} + \wp) = \frac{4}{9} \, \kappa^2$$ 
where $\wp$ is the Weierstrass function with invariants 
$$g_2 = \frac{4}{27} (9 - 8 k^2) =  \frac{4}{27} (1 + 8 \lambda^2)$$
$$g_3 = \frac{8}{729} (8 k^4 - 36 k^2 + 27) = \frac{8}{729} (8 \lambda^4 + 20 \lambda^2 - 1).$$ 
\end{theorem} 

\begin{proof} 
Transform the differential equation of Theorem \ref{d'} by defining the new function 
$$p = \frac{4}{9} \kappa^2 (1 - \dd)^{-1} - \frac{1}{3}.$$
It turns out that $p$ satisfies the differential equation 
$$(p')^2 = 4 p^3 - g_2 p - g_3$$
in which the coefficients $g_2$ and $g_3$ are as advertised; as $p$ has a pole at $0$, we conclude that $p$ is indeed the Weierstrass function $\wp$ with $g_2$ and $g_3$ as invariants. 
\end{proof} 

\medbreak 

The foregoing proof is condensed from [4]; see [5] for the original argument, which is of independent interest. 

\medbreak 

We have now arrived at the elliptic function $\dd$ proper: namely 
$$\dd = 1 - \frac{\frac{4}{9} \kappa^2}{\frac{1}{3} + \wp}$$
in terms of its coperiodic Weierstrass function $\wp$. This expression makes it plain that the elliptic function $\dd$ has order two; moreover, that $\dd$ is both even (in the usual sense that $\dd(-z) = \dd(z))$) and `real' (in the sense that $\overline{\dd (\overline{z})} = \dd(z)$). 

\medbreak We can say more. The Weierstrass function $\wp$ has discriminant 
$$g_2^3 - 27 g_3^2 = \frac{16^3}{27^3} k^6 (1 - k^2) > 0$$
and its invariants are real. Consequently, $\wp$ has a rectangular period lattice: we may take as fundamental periods $2 \omega = 2 K$ and $2 \omega' = 2 \ii K'$ with $K > 0$ and $K' > 0$; of course, these are also fundamental periods of $\dd$. As period parallelogram, we may take either the rectangle whose vertices are 
$$0, \, 2 K, \, 2K + 2 \ii K', \, 2 \ii K'$$
or the rectangle whose edges have midpoints 
$$\pm \, K, \, \pm \, \ii K'.$$

\medbreak 

The intersection of these two rectangles we shall call the quarter-rectangle $Q$. Along the perimeter of this rectangle the Weierstrass function $\wp$ is real-valued, its values decreasing strictly from $+ \infty$ to $- \infty$ as the perimeter is traced in the order 
$$0 \to K \to K + \ii K' \to \ii K' \to 0.$$ 
Here of course, $0$ represents the pole lattice of $\wp$. The three nonzero vertices represent the three midpoint lattices, where $\wp'$ is zero; the values assumed by $\wp$ at these points satisfy 
$$\wp (K) > \wp(K + \ii K') > \wp(\ii K')$$ 
with $\wp(K)$ positive and $\wp(\ii K')$ negative. 

\medbreak 

In the next three sections, we address three separate structural features of the elliptic function $\dd$: we determine the precise location of its poles; we determine the precise values of $\dd$ at the vertices of the quarter-rectangle $Q$, these vertices representing the lattices where the derivative $\dd'$ vanishes; and we determine the precise values of the fundamental periods $2 K$ and $2 \ii K'$. 

\medbreak 

\section*{Poles}

\medbreak 

In this section, we investigate the poles of $\dd$ and pin down their precise locations. 

\medbreak 

The formula $(1 - \dd)\, (\frac{1}{3} + \wp) = \frac{4}{9} \, \kappa^2$ of Theorem \ref{ddwp} makes it plain that the elliptic function $\dd$ has poles precisely where its coperiodic Weierstrass function $\wp$ assumes the value $- 1/3$; thus our task is clear. 

\medbreak 

We begin our search at the destination, looking among the points of order three modulo periods of $\wp$. Let $z$ be such a point: equivalently, let $- 2 z \equiv z \nequiv 0$ where $\equiv$ signifies congruence modulo periods. In the duplication formula 
$$\wp (2 z) + 2 \wp (z) = \frac{1}{4} \, \frac{\wp''(z)^2}{\wp'(z)^2}$$
for $\wp$, take account of the fact that $\wp$ is even and of the derivative formulae 
$$\wp''(z) = 6 \wp(z)^2 - \frac{1}{2} g_2$$ 
$$\wp'(z)^2 = 4 \wp(z)^3 - g_2 \wp(z) - g_3$$ 
to deduce that 
$$3 \wp(z) = \frac{1}{4} \, \frac{36 \wp(z)^4 - 6 g_2 \wp(z)^2 + \frac{1}{4} g_2^2}{4 \wp(z)^3 - g_2 \wp(z) - g_3}$$
and conclude that $w = \wp(z)$ satisfies the quartic equation 
$$w^4 - \frac{1}{2} g_2 w^2 - g_3 w - \frac{1}{48} g_2^2 = 0.$$ 
This quartic has real coefficients and its discriminant $\Delta$ is given by 
$$27 \Delta = - (g_2^3 - 27 g_3^2)^2 < 0$$
whence it has two real zeros and a conjugate pair of nonreal zeros; moreover, the real zeros have opposite sign since the product of all four zeros is $- g_2^2 / 48 < 0$. 

\medbreak 

The following result was first established using theta functions and conformal mapping theory in [5]; our argument using the coperiodic Weierstrass function is adapted from [4]. 

\medbreak 

\begin{theorem} \label{poles}
The elliptic function $\dd$ has poles at the points $\pm \frac{2}{3} \ii K'.$
\end{theorem} 

\begin{proof} 
Let $z$ be one of the two indicated points. As $z$ has order three modulo periods, $w = \wp(z)$ is a zero of the quartic that appears just prior to this Theorem. Recall that the values of $\wp$ are negative along the left-hand edge $(0, \ii K')$ of the quarter-rectangle $Q$. It follows that $\wp(z)$ is the unique negative zero of the quartic 
$$w^4 - \frac{1}{2} g_2 w^2 - g_3 w - \frac{1}{48} g_2^2.$$ 
With $g_2$ and $g_3$ as given in Theorem \ref{ddwp} it is readily checked that this quartic vanishes at $-1/3$. Thus  $\wp(z) = -1/3$ and so $\dd$ has poles at  $\pm \frac{2}{3} \ii K'$ as claimed. 
\end{proof} 

\medbreak 

Each of these poles is simple. This may be seen indirectly: on the one hand, the elliptic function $\dd$ has order two; on the other hand, both of these poles lie in the fundamental rectangle with edge midpoints at $\pm K, \, \pm \ii K'$. It may also be seen directly: for instance, $\wp = -1/3$ implies $\wp' \neq 0$, so the zeros of $\frac{1}{3} + \wp$ are simple. 

\medbreak 

The full list of order-three points in the rectangle having edge midpoints at $\pm K, \, \pm \ii K'$ is 
$$\pm \frac{2}{3} K, \, \pm \frac{2}{3} \ii K', \, \pm \frac{2}{3} (K + \ii K'), \, \pm \frac{2}{3} (K - \ii K').$$
Here, $\wp(\pm \frac{2}{3} K)$ is the positive zero of the quartic in the proof of Theorem \ref{poles} while $\wp(\frac{2}{3} K + \frac{2}{3}\ii K')$ and $\wp(\frac{2}{3} K - \frac{2}{3} \ii K')$ constitute its conjugate pair of nonreal zeros. 

\medbreak 

\section*{Extrema}

\medbreak 

In this section, we make explicit the critical values of $\dd$: that is, the values that $\dd$ assumes at the points at which its derivative is zero. So as not to break the run of one-word section headings (and on the authority of Hille [3] page 284) we call these points extrema of $\dd$. 

\medbreak 

Recall from Theorem \ref{ddwp} the explicit formula for $\dd$ in terms of $\wp$; by differentiation, it follows that 
$$\frac{\dd'}{1 - \dd} = \frac{\wp'}{\frac{1}{3} + \wp}.$$
Consequently, the extrema of $\dd$ are precisely the poles of $\wp$ together with the extrema of $\wp$. The poles of $\wp$ are of course represented by $0$; the corresponding critical value of $\dd$ is $1$. Our task in this section is therefore to evaluate $\dd$ on the midpoint lattices of $\wp$: accordingly, we evaluate $\dd$ at the vertices $K, \, K + \ii K', \, \ii K'$ of the quarter-rectangle $Q$. 

\medbreak 

From the differential equation 
$$(\dd ')^2 = \frac{4}{9} (1 - \dd)\,(\dd^3 + 3 \, \dd^2 - 4 \lambda^2)$$ 
of Theorem \ref{d'}, the critical values in question are therefore the three zeros of the cubic $g$ defined by 
$$g(x) = x^3 + 3 x^2 - 4 \lambda^2.$$ 
We proceed to give concrete expression to these zeros and hence to the critical values of $\dd$. 

\medbreak 

Observe at once that 
$$g(-3) = - 4\lambda^2, \; g(-2) = 4(1 - \lambda^2), \; g(0) = - 4\lambda^2, \; g(1) = 4(1 - \lambda^2).$$
Consequently, the cubic $g$ has three real zeros: following [5] we shall denote these zeros by 
$$x_1 \in (0, 1), \; - x_2 \in (-2, 0), \; - x_3 \in (-3, -2)$$
so that $x_1, x_2, x_3$ are all strictly positive and 
$$-3 < - x_3 < - 2 < - x_2 < 0 < x_1 < 1.$$ 

\medbreak 

The form of the cubic $g$ suggests a trigonometric approach. In the equation 
$$x^3 + 3 x^2 - 4 \lambda^2 = 0$$
let $y = 1/x$ so that 
$$1 + 3 y - 4 \lambda^2 y^3 = 0$$ 
and then multiply throughout by $\lambda$: with $z = \lambda y = \lambda / x$ we arrive at 
$$4 z^3 - 3 z = \lambda.$$ 

\medbreak 

Specify the angle $\aaa \in (0, \pi/6)$ by the condition 
$$\lambda = \cos 3 \aaa$$
or the equivalent condition that $\sin 3 \aaa$ equal the modulus $\kappa$. 

\medbreak 

One root of the cubic equation 
$$4 z^3 - 3 z = \cos 3 \aaa$$ 
is of course $z = \cos \aaa$; from $x = \lambda / z$ we deduce the following explicit formulae for the positive critical value $x_1$. 

\medbreak 

\begin{theorem} \label{x1}
The positive critical value of $\dd$ is 
$$\dd(K) = x_1 = 4 \cos^2 \aaa - 3 = 2 \cos 2 \aaa - 1.$$ 
\end{theorem} 

\begin{proof} 
Trigonometric triplication and duplication: substitution into $x = \lambda / z$ of $z = \cos \aaa$ and 
$$\lambda = \cos 3 \aaa = 4 \cos^3 \aaa - 3 \cos \aaa = (4 \cos^2 \aaa - 3) \cos \aaa = (2 \cos 2 \aaa - 1) \cos \aaa.$$
\end{proof} 

\medbreak 

The other two roots of the cubic in $z$ are $\cos (\aaa \pm 2 \pi/3)$; these lead to the following formulae for the negative critical values of $\dd$. 

\medbreak 

\begin{theorem} \label{x23}
The negative critical values $\dd(K + \ii K') = - x_2$ and $\dd(\ii K') = - x_3$ of $\dd$ are given by 
$$x_2 = 1 + 2 \cos (2 \aaa + \pi/3) = 1 + \cos 2 \aaa - \sqrt{3} \sin 2 \aaa$$
and 
$$x_3 = 1 + 2 \cos(2\aaa - \pi/3) = 1 + \cos 2 \aaa + \sqrt{3} \sin 2 \aaa.$$ 
\end{theorem} 

\begin{proof} 
From the standard trigonometric formula 
$$\cos A + \cos B = 2 \cos \tfrac{1}{2} (A + B) \, \cos \tfrac{1}{2} (A - B)$$ 
it follows that 
$$\cos 3 \aaa + \cos (\aaa - 2 \pi / 3) = 2 \cos (2 \aaa - \pi / 3) \, \cos (\aaa + \pi / 3)$$ 
while 
$$\cos (\aaa + \pi / 3) = - \cos (\aaa - 2 \pi / 3)$$ 
so that 
$$\cos 3 \aaa = - \cos (\aaa - 2 \pi / 3) \, \big[ 1 + 2 \cos (2 \aaa - \pi / 3) \big].$$ 
Hence $x_3$; likewise $x_2$. 
\end{proof} 

\medbreak

As a bonus, we may now derive corresponding formulae for the midpoint values
$$\wp(K) > \wp(K + \ii K') > \wp(\ii K')$$
 of the coperiodic Weierstrass function $\wp$, in the same trigonometric terms. 

\medbreak 

\begin{theorem} \label{wp}
The midpoint values of $\wp$ are given by 
$$\tfrac{1}{3} + \wp (K) = \tfrac{\kappa^2}{9} \csc^2 \aaa$$
$$\tfrac{1}{3} + \wp (K + \ii K')) = \tfrac{\kappa^2}{9} \sec^2 (\aaa + \tfrac{\pi}{6})$$
$$\tfrac{1}{3} + \wp (\ii K') = \tfrac{\kappa^2}{9} \sec^2 (\aaa - \tfrac{\pi}{6}).$$
\end{theorem} 

\begin{proof} 
Into Theorem \ref{ddwp} substitute the results of Theorem \ref{x1} and Theorem \ref{x23}, using $\dd(K) = x_1$, $\dd(K + \ii K') = - x_2$ and $\dd(\ii K') = - x_3$. For $\wp(K)$ argue by trigonometric duplication that  
$$1 - \dd(K) = 1 - x_1 = 2 - 2 \cos 2 \aaa = 4 \sin^2 \aaa,$$
arguing similarly for $\wp(K + \ii K')$ and $\wp(K')$. 
\end{proof} 

\medbreak 

Here, recall that $\kappa = \sin 3 \aaa$; so, for example,  
$$\tfrac{1}{3} + \wp (K) = \tfrac{1}{9} (3 - 4 \sin^2 \aaa)^2.$$

\medbreak 

\medbreak 

\section*{Periods}

\medbreak 

In this section, we make explicit the fundamental periods $2 K$ and $2 \ii K'$ of $\dd$ and of $\wp$: we render them first in terms of the positive real numbers $x_1, x_2, x_3$ and thereafter in terms of the angle $\aaa$. 

\medbreak 

We shall make free use of the formulae recorded by Greenhill: for ease of reference, we present the results of cases (ii) and (iii) in Section 66 of [2] with minor modifications. 

\medbreak 

Let the quartic 
$$X = (x - \alpha)(x - \beta)(x - \gamma)(x - \delta)$$
have real zeros 
$$\alpha > \beta > \gamma > \delta.$$ 
Write 
$$M = \tfrac{1}{2} \sqrt{(\alpha - \gamma)(\beta - \delta)};$$
also write 
$$k = \sqrt{\frac{(\alpha - \beta)(\gamma - \delta)}{(\alpha - \gamma)(\beta - \delta)}}$$
and write 
$$k' = \sqrt{\frac{(\beta - \gamma)(\alpha - \delta)}{(\alpha - \gamma)(\beta - \delta)}}$$
noting that 
$$k^2 + k'^{\,2} = 1.$$

\medbreak 

Case (ii) in Section 66 of [2] asserts that if $\alpha > x > \beta$ then 
$$\int_{\beta}^x \frac{M {\rm d}x}{\sqrt{ - X}} = {\rm sn}^{-1} \sqrt{\frac{(\alpha - \gamma)(x - \beta)}{(\alpha - \beta)(x - \gamma)}}$$
where ${\rm sn} = {\rm sn}(\bullet ; k)$ is the Jacobian `sinus amplitudinus' with modulus $k$; in particular, 
$$\int_{\beta}^{\alpha} \frac{M {\rm d}x}{\sqrt{ - X}} = {\rm sn}^{-1} 1 = K(k)$$
where 
$$K(k) = \tfrac{1}{2} \pi F(\tfrac{1}{2}, \tfrac{1}{2}; 1; k^2)$$
is the corresponding complete elliptic integral. 

\medbreak 

Case (iii) in Section 66 of [2] asserts that if $\beta > x > \gamma$ then 
$$\int_x^{\beta} \frac{M {\rm d}x}{\sqrt{X}} = {\rm sn}^{-1} \sqrt{\frac{(\alpha - \gamma)(\beta - x)}{(\beta - \gamma)(\alpha - x)}}$$
where ${\rm sn} = {\rm sn}(\bullet ; k')$ is the Jacobian `sinus amplitudinus' with modulus $k'$; in particular, 
$$\int_{\gamma}^{\beta} \frac{M {\rm d}x}{\sqrt{X}} = {\rm sn}^{-1} 1 = K(k')$$
where 
$$K(k') = \tfrac{1}{2} \pi F(\tfrac{1}{2}, \tfrac{1}{2}; 1; k'{\,^2})$$
is the corresponding complete elliptic integral. 

\medbreak 

Recall from Theorem \ref{d'} that $\dd$ satisfies the differential equation 
$$(\dd ')^2 = \frac{4}{9} (1 - \dd)\,(\dd^3 + 3 \, \dd^2 - 4 \lambda^2).$$ 
Accordingly, we apply the foregoing results to the quartic $X$ for which 
$$\alpha = 1, \, \beta = x_1, \, \gamma = - x_2, \, \delta = - x_3$$
so that 
$$M = \tfrac{1}{2} \sqrt{(1 + x_2)(x_3 + x_1)}$$
along with 
$$k^2 = \frac{(1 - x_1) (x_3 - x_2)}{(1 + x_2)(x_3 + x_1)}$$
and 
$$k'{\,^2} = \frac{(1 + x_3) (x_1 + x_2)}{(1 + x_2)(x_1 + x_3)}\, .$$
\medbreak 
\noindent 
Here, note that both $0 < k < 1$ and $0 < k' < 1$: this is at once clear for $k$; for $k'$ it follows from the identity $k^2 + k'^{\, 2} = 1$. 

\medbreak 

Substitution of $x_1$, $x_2$ and $x_3$ from Theorem \ref{x1} and Theorem \ref{x23} yields the explicit formula 
$$M^2 = 2 \sqrt3 \cos^2 (\aaa + \pi / 6) \cos(2 \aaa - \pi / 6).$$

\medbreak 

\begin{theorem} \label{K} 
The fundamental period $2 K$ of $\dd$ is given by 
$$2 K = \frac{3 \pi}{2 M} F(\tfrac{1}{2}, \tfrac{1}{2}; 1; k^2)$$
where 
$$k^2 = \frac{\sin^2 \aaa \sin 2 \aaa}{\cos^2 (\aaa + \pi / 6) \cos(2 \aaa - \pi / 6)}\,.$$
\end{theorem} 

\begin{proof} 
Integrate along the lower edge of the quarter-rectangle $Q$. As $\dd = 1 - \frac{4}{9} \kappa^2 / (\frac{1}{3} + \wp)$ decreases along the interval $(0, K)$ and satisfies the differential equation 
$$(\dd ')^2 = \frac{4}{9} (1 - \dd)\,(\dd^3 + 3 \, \dd^2 - 4 \lambda^2)$$ 
it follows that  
$$\dd' = - \frac{2}{3} \sqrt{(1 - \dd)\,(\dd^3 + 3 \, \dd^2 - 4 \lambda^2)}$$
along the same interval, whence by integration 
$$- \frac{3}{2} \int_{\dd(0)}^{\dd(K)} \frac{{\rm d} x}{\sqrt{(1 - x)(x^3 + 3 x^2 - 4 \lambda^2)}} = \int_0^K 1$$
or 
$$K = \frac{3}{2} \int_{x_1}^1 \frac{{\rm d} x}{\sqrt{(1 - x)(x^3 + 3 x^2 - 4 \lambda^2)}}\,.$$ 
Now case (ii) in Section 66 of [2] applies. To complete the argument, note that 
$$1 - x_1 = 2 - 2 \cos 2 \aaa = 4 \sin^2 \aaa$$
and 
$$1 + x_2 = 2 + 2 \cos (2 \aaa + \pi / 3) = 4 \cos^2 (\aaa + \pi / 6)$$ 
whence 
$$\frac{1 - x_1}{1 + x_2} = \frac{\sin^2 \aaa}{\cos^2 (\aaa + \pi / 6)}$$
while similar arguments yield 
$$\frac{x_3 - x_2}{x_3 + x_1} = \frac{\sin 2 \aaa}{\cos (2 \aaa - \pi / 6)}.$$
\end{proof} 

\medbreak 

To determine the value of $K'$ we integrate up the right edge of the quarter-rectangle $Q$. 

\medbreak 

\begin{theorem} \label{K'}
The fundamental period $2 K'$ of $\dd$ is given by 
$$2 K' = \frac{3 \pi}{2M} F(\tfrac{1}{2}, \tfrac{1}{2}; 1; k'{\,^2})$$
where 
$$k'^{\, 2} = \frac{\cos^2 (\aaa - \pi / 6)\cos ( 2 \aaa + \pi / 6)}{\cos^2 (\aaa + \pi / 6) \cos(2 \aaa - \pi / 6)}\,.$$
\end{theorem} 

\begin{proof} 
Similar to that for Theorem \ref{K}. Define a new function ${\rm d}$ by 
$${\rm d}(y) = \dd (K + \ii y)$$
so that ${\rm d}$ satisfies the differential equation 
$$({\rm d}')^2 = \frac{4}{9}  (1 - {\rm d})(4 \lambda^2 - 3 \, \dd^2 - \dd^3)$$
whence  
$${\rm d}' = - \frac{2}{3} \sqrt{(1 - {\rm d})(4 \lambda^2 - 3 \, \dd^2 - \dd^3)}$$ 
along the interval $(0, K')$ and integration yields 
$$K' = \frac{3}{2} \int_{- x_2}^{x_1} \frac{{\rm d} x}{\sqrt{(1 - x)(4 \lambda^2 - 3 x^2 - x^3)}}\, .$$
Apply case (iii) in Section 66 of [2] to complete the argument. 
\end{proof} 

\medbreak 

It is perhaps worth observing that integration up the left edge of the quarter-rectangle $Q$ encounters the pole of $\dd$ at $\tfrac{2}{3} \ii K'$.  Again, [2] Section 66 applies: integration from $0$ up to the pole uses case (i) while integration from the pole up to $\ii K'$ uses case (v); the resulting expressions for $2 K'$ do not present themselves as complete elliptic integrals. 

\medbreak 

Incidentally, we may also locate the zeros of $\dd$: in fact, the zero $K + \ii L'$ that lies on the right edge of the quarter-rectangle $Q$ may be found by integration as in the proof of Theorem \ref{K'} but along $(0, L')$; again, the expression that results is not a complete elliptic integral. 

\medbreak 

To close this section, we record the value of the quotient $K' / K$ in terms of the complementary moduli $k$ and $k'$.

\medbreak 

\begin{theorem} \label{K/K'}
$$\frac{K'}{K} = \frac{F(\tfrac{1}{2}, \tfrac{1}{2}; 1; k'^{\, 2})}{F(\tfrac{1}{2}, \tfrac{1}{2}; 1; k^2)} = \frac{F(\tfrac{1}{2}, \tfrac{1}{2}; 1; 1 - k^2)}{F(\tfrac{1}{2}, \tfrac{1}{2}; 1; k^2)} \,.$$ 
\end{theorem} 

\begin{proof} 
Theorem \ref{K} and Theorem \ref{K'} involve the same constant $M$.
\end{proof} 

\medbreak 

\section*{Hypergeometrics} 

\medbreak 

Here we shall see that a curious hypergeometric identity falls out of our analysis. 

\medbreak 

Note first that the function $\sss = \sin \circ \phi$ is not the restriction of an elliptic function, but its square is such a restriction: in fact, the equation 
$$4 \kappa^2 \sss^2 = (1 - \dd)(2 + \dd)^2$$
from Theorem \ref{Py} makes it plain that $\sss^2$ extends to (or abusively `is') an elliptic function, with third-order poles where $\dd$ has its poles. 

\medbreak 

The zeros of $\sss^2$ come in two varieties: the zeros of $1 - \dd$ and the zeros of $2 + \dd$. The zeros of $1 - \dd$ coincide with the (double) poles of $\wp$ at the points congruent to $0$ modulo periods, according to Theorem \ref{ddwp}; of course, these zeros of $\sss^2$ are expected. The zeros of $2 + \dd$ fall into two congruence classes: $\dd$ takes the value $-2$ at a point $z_0$ on the upper edge of the quarter-rectangle $Q$ because $- x_3 < - 2 < - x_2$; the zeros of $2 + \dd$ are then the points congruent to $\pm z_0$ modulo periods. Each of these zeros of $\sss^2$ has even order. 

\medbreak 

Now, on account of Theorem \ref{poles}, the open band 
$$B = \{ z \in \mathbb{C}: |{\rm Im} z | < \tfrac{2}{3} K' \}$$
about the real axis excludes the poles of $\sss^2$. Within $B$,  the function $\sss^2$ is holomorphic and has even-order zeros (at the multiples of $2 K$) so $\sss^2$ has two holomorphic square-roots: one of these square-roots extends the original $\sss = \sin \circ \phi$ and will be called simply $\sss$; the other square-root is the negative of this one. 

\medbreak 

Now restrict attention to $\sss$ along the real line. We claim that $\sss (K) = 1$: in fact, 
$$4 (1 - \kappa^2 \sss^2) = \dd^3 + 3 \, \dd^2$$
and the positive critical value $\dd(K) = x_1$ satisfies 
$$\dd(K)^3 + 3 \, \dd(K)^2 = 4 \lambda^2$$ 
so that 
$$1 - \kappa^2 \sss(K)^2 = \lambda^2 = 1 - \kappa^2$$
and therefore $\sss(K)^2 = 1$; as $\sss'(0) = \ccc(0) \dd(0) = 1 > 0$ and $\sss$ has $2 K$ as its first positive zero, we conclude that $\sss (K) = 1$ as claimed. 

\medbreak 

More: on the interval $(0, K)$ the function $\dd$ is positive and strictly decreasing, whence $\dd^3 + 3 \, \dd^2$ exceeds the value $4 \lambda^2$ that it achieves at $K$; on the same interval, it follows that $4 (1 - \kappa^2 \sss^2) > 4 (1 - \kappa^2)$ whence $\sss^2 < 1$. In particular, $\sss$ first takes the value $1$ at $K$. 

\medbreak 

Recall the original definition of $\phi$ as a local inverse to $f$ defined by the rule 
$$f (T) = \int_0^T F(\tfrac{1}{3}, \tfrac{2}{3} ; \tfrac{1}{2} ; \kappa^2 \sin^2 t) \, {\rm d}t.$$ 
At its introduction, we noted that $f$ is actually defined on the whole real line; moreover, $f'$ is strictly positive and periodic there, so the inverse $\phi$ is also defined on the whole real line and the equation 
$$u = \int_0^{\phi(u)} F(\tfrac{1}{3}, \tfrac{2}{3} ; \tfrac{1}{2} ; \kappa^2 \sin^2 t) \, {\rm d}t$$ 
is valid for all real $u$. It now follows easily from $1 = \sss (K) = \sin \phi (K)$ that 
$$\phi (K) = \tfrac{1}{2} \pi$$ 
and therefore that 
$$K =  \int_0^{\tfrac{1}{2} \pi} F(\tfrac{1}{3}, \tfrac{2}{3} ; \tfrac{1}{2} ; \kappa^2 \sin^2 t) \, {\rm d}t.$$
Here, termwise integration of the hypergeometric series produces the familiar identity 
$$\int_0^{\tfrac{1}{2} \pi} F(a, b ; \tfrac{1}{2} ; \kappa^2 \sin^2 t)\, {\rm d} t = \tfrac{1}{2} \pi \, F(a, b ; 1 ; \kappa^2)$$
in which we take $a = \tfrac{1}{3}$ and $b = \tfrac{2}{3}$ to secure the formula 
$$K =  \tfrac{1}{2} \pi \, F(\tfrac{1}{3}, \tfrac{2}{3} ; 1 ; \kappa^2).$$
In context, this is arguably the most natural formula for the half-period $K$ of $\dd$. 

\medbreak 

Comparison with Theorem \ref{K} at once yields the following curious outcome. 

\medbreak 
\noindent 
{\bf Conclusion}: $F(\tfrac{1}{3}, \tfrac{2}{3} ; 1 ; \bullet)$ and $F(\tfrac{1}{2}, \tfrac{1}{2} ; 1 ; \bullet)$ are related by the identity 
$$F(\tfrac{1}{3}, \tfrac{2}{3} ; 1 ; \kappa^2) = \frac{3}{2 M} \, F(\tfrac{1}{2}, \tfrac{1}{2}; 1; k^2)$$ 
wherein 
$$\kappa = \sin 3 \aaa = \sin \aaa \, (3 - 4 \sin^2 \aaa),$$
the modulus $k$ is given by 
$$k^2 = \frac{\sin^2 \aaa \sin 2 \aaa}{\cos^2 (\aaa + \pi / 6) \cos(2 \aaa - \pi / 6)}$$ 
and the multiplier $3 / 2 M$ is given by 
$$4 M^2 = \sqrt{3} \, (2 + \cos 2 \aaa - \sqrt{3} \sin 2 \aaa) (\sqrt{3} \cos 2 \aaa + \sin 2 \aaa).$$

\medbreak 

Let us at once remove any mystery from this curious identity: it is none other than a famous identity recorded by Ramanujan in his second notebook (on page 258 in the Tata Institute publication) and proved by Berndt, Bhargava and Garvan; we quote it from [1] in the following gently massaged form. 

\medbreak 
\noindent 
{\bf Theorem 5.6}. If $0 < p < 1$ then 
$$F(\tfrac{1}{3}, \tfrac{2}{3} ; 1 ; \beta) = \frac{1 + p + p^2}{\sqrt{1 + 2 p}} \, F(\tfrac{1}{2}, \tfrac{1}{2}; 1; \alpha)$$
where 
$$\alpha = p^3 \frac{2 + p}{1 + 2 p} \; \; {\rm and} \; \; \beta = \frac{27}{4} \frac{p^2 (1 + p)^2}{(1 + p + p^2)^3}\,.$$ 

\medbreak 

We proceed to demonstrate this identity of identities. 

\medbreak 

First, we match $\kappa^2$ and $\beta$: that is, we relate $\aaa$ and $p$ by the requirement  
$$\sin^2 \aaa \, (3 - 4 \sin^2 \aaa)^2 = \kappa^2 = \beta = \frac{27}{4} \frac{p^2 (1 + p)^2}{(1 + p + p^2)^3}.$$ 
Writing $S$ for $\sin^2 \aaa$ presents us with the cubic equation  
$$S \, (3 - 4 S)^2 = \frac{27}{4} \frac{p^2 (1 + p)^2}{(1 + p + p^2)^3}$$ 
for $S$, with three solutions: namely, 
$$S = \frac{3}{4 (1 + p + p^2)}, \; S = \frac{3 p^2}{4 (1 + p + p^2)}, \; S = \frac{3 (1 + p)^2}{4 (1 + p + p^2)}; $$
the outer solutions here satisfy $S > 1/4$ while the middle solution satisfies $0 < S < 1/4$. Note that this middle solution $S$ satisfies 
$$\frac{1}{S} = \frac{4}{3} \Big(1 + \frac{1}{p} + \frac{1}{p^2}\Big)$$
so that $S$ increases strictly with $p$ and the assignment $(0, 1) \to (0, 1/4) : p \mapsto S$ is bijective. 

\medbreak 

Given $0 < p < 1$ we therefore define $0 < \aaa < \pi / 6$ by the rule 
$$\sin \aaa = \frac{\sqrt 3} {2} \, \frac{p}{\sqrt{1 + p + p^2}}$$
whence 
$$\cos \aaa = \frac{1}{2} \, \frac{2 + p}{\sqrt{1 + p + p^2}}\,.$$

\medbreak 

With this correspondence, it follows that 
$$\sin 2 \aaa = \frac{\sqrt{3}}{2} \, \frac{p (2 + p)}{1 + p + p^2}$$
and 
$$\cos 2 \aaa = \frac{1}{2} \, \frac{2 + 2 p - p^2}{1 + p + p^2} \, .$$

\medbreak 

It is now a moderately pleasing exercise to complete the identification by verifying that 
$$4 M^2 = 9 \, \frac{1 + 2 p}{(1 + p + p^2)^2}$$
and that 
$$k^2 = p^3  \frac{2 + p}{1 + 2 p}\,.$$

\medbreak 

Of course, this is only the beginning: for example, companion identities arise when complementary moduli are considered. 

\bigbreak 

The elliptic function $\dd$ thus leads to a new derivation of a famous Ramanujan identity and indeed provides a convenient entry into signature three, further cementing its place in the Ramanujan theory of elliptic functions to alternative bases. 

\bigbreak

\begin{center} 
{\small R}{\footnotesize EFERENCES}
\end{center} 
\medbreak 

[1] B.C. Berndt, S. Bhargava, and F.G. Garvan, {\it Ramanujan's theories of elliptic functions to alternative bases}, Transactions of the American Mathematical Society {\bf 347} (1995) 4163-4244. 

\medbreak 

[2] A.G. Greenhill, {\it The Applications of Elliptic Functions}, Macmillan and Company (1892); Dover Publications (1959). 

\medbreak 

[3] E. Hille, {\it Ordinary Differential Equations in the Complex Domain}, Wiley-Interscience (1976); Dover Publications (1997). 

\medbreak 

[4] P.L. Robinson, {\it Elliptic functions from $F(\frac{1}{3}, \frac{2}{3} ; \frac{1}{2} ; \bullet)$}, arXiv 1907.09938 (2019). 

\medbreak 

[5] Li-Chien Shen, {\it On the theory of elliptic functions based on $_2F_1(\frac{1}{3}, \frac{2}{3} ; \frac{1}{2} ; z)$}, Transactions of the American Mathematical Society {\bf 357}  (2004) 2043-2058. 

\medbreak 

\medbreak

\end{document}